\documentclass[11pt]{article}
\usepackage{amsthm, amsmath, amssymb, amsfonts, url, booktabs, tikz, setspace, fancyhdr, bm, mathrsfs}
\usepackage{hyperref}
\usepackage{geometry}
\usepackage{hyperref, enumerate}
\usepackage[shortlabels]{enumitem}
\usepackage[babel]{microtype}
\usepackage[english]{babel}
\usepackage[capitalise]{cleveref}
\usepackage{comment}
\usepackage{bbm}
\usepackage{csquotes}
\usepackage{mathabx}
\usepackage{tikz}
\usepackage{graphicx}
\usepackage{float}
\usepackage{amsmath}

\counterwithin{figure}{section}

\newtheorem{theorem}{Theorem}[section]

\newtheorem{lemma}[theorem]{Lemma}

\newtheorem{claim}[theorem]{Claim}

\theoremstyle{definition}

\newtheorem*{defn-non}{Definition}

\usepackage[linesnumbered, ruled]{algorithm2e}
\SetKwRepeat{Do}{do}{while}%

\newenvironment{poc}{\begin{proof}[Proof of the claim]}{\end{proof}}

\usepackage{todonotes}
\newcommand*{\abs}[1]{\lvert#1\rvert}

\newcommand{\cG}{\mathcal{G}}
\newcommand{\cH}{\mathcal{H}}

\newcommand{\VC}{\operatorname{VC}}

\title{Bounded VC-dimension implies the Erd\H{o}s--Rado sunflower conjecture}
\author{
Gennian Ge\thanks{School of Mathematical Sciences, Capital Normal University, Beijing, China. Email: gnge@zju.edu.cn. Gennian Ge is supported by the National Key Research and Development Program of China under Grant 2025YFC3409900, the National Natural Science Foundation of China under Grant 12231014, and Beijing Scholars Program.}
\and
Jian Wang\thanks{School of Mathematics, Sichuan University,
Chengdu, China. Email: wangjianmath01@scu.edu.cn. Jian Wang is supported by National Natural Science
Foundation of China Grant no. 12471316.}
\and
Zixiang Xu\thanks{School of Mathematical Sciences, Zhejiang University, Hangzhou, China. Email: zixiangxu@zju.edu.cn.}
\and
Xiaochen Zhao\thanks{School of Mathematical Sciences, Capital Normal University, Beijing, China. Email: 2250501013@cnu.edu.cn.}
}
\date{}

\begin{document}

\maketitle

\begin{abstract}
The Erd\H{o}s--Rado sunflower conjecture asserts that, for every fixed integer \(r\ge 3\), there is a constant \(K(r)\) such that every \(\ell\)-uniform family with more than \(K(r)^{\ell}\) members contains an \(r\)-sunflower. We prove this conjecture for families of bounded VC-dimension.
\end{abstract}

\section{Introduction}

A collection of distinct sets \(H_{1},\ldots,H_{r}\) is an \(r\)-\emph{sunflower}, or a \(\Delta\)-system, if every two of its members have the same intersection. Thus there is a set \(C\), called the \emph{kernel}, such that \(H_{i}\cap H_{j}=C\) whenever \(i\ne j\), the sets \(H_{i}\setminus C\) are the petals and are pairwise disjoint. A family is \(\ell\)-\emph{bounded} if each of its members has size at most \(\ell\), and is \(\ell\)-uniform if each of its members has size exactly \(\ell\). This simple intersection pattern is one of the fundamental configurations in extremal set theory. For integers \(r\ge 3\) and \(\ell\ge 1\), let \(f_{r}(\ell)\) be the maximum size of an \(\ell\)-uniform family containing no \(r\)-sunflower. In 1960, Erd\H{o}s and Rado~\cite{ER60} proved the classical sunflower lemma \(f_{r}(\ell)\le (r-1)^{\ell}\ell!\), which have become standard reduction tools across combinatorics and theoretical computer science \cite{AB87,CKR22,DV25,FLSZ19,GMR13,LLZ18,LMMPZ22,NSR18,RA20,Rao26,R85}, and the remarkable persistence of the problem reflects the difficulty of forcing even this rigid local structure from a purely global hypothesis on the size of a family.

Erd\H{o}s and Rado~\cite{ER60} also conjectured that the factorial term can be removed: there should be a constant \(K(r)\), depending only on \(r\), for which
\[
f_{r}(\ell)\le K(r)^{\ell}.
\]
The conjecture has been one of the favorite open problems of Erd\H{o}s~\cite{E81}, which remains open in general and also stands as one of the most important problem in the field of extremal combinatorics~\cite{ErdosProblem}. In a major breakthrough, Alweiss, Lovett, Wu, and Zhang~\cite{ALWZ21} obtained a bound of the form \(\bigl(Cr^{3}\log \ell\log\log \ell\bigr)^{\ell}\), Bell, Chueluecha, and Warnke~\cite{BCW21} subsequently sharpened the best general estimate to \(\bigl(Cr\log \ell\bigr)^{\ell}\). These advances introduced and refined the spread viewpoint that now plays a central role in extremal set theory
\cite{BDGGL26,FHIKLMP25,FK25,INST25,KLS25,KN24,KZD24,WX26}.

For a set system \(\cH\), a set \(D\) is \emph{shattered} if \(\{H\cap D:H\in\cH\}=2^{D}\), and the VC-dimension \(\VC(\cH)\) is the largest size of such a set. This parameter measures the complexity of the traces of \(\cH\). Its basic counting consequence is the Sauer--Shelah lemma \cite{Sauer72,Shelah72,VC71}, which gives the largest size of a set systems with bounded VC-dimension.

Bounded VC-dimension has repeatedly made difficult extremal conjectures accessible. Fox, Pach, and Suk~\cite{FPS19} initiated the study of the Erd\H{o}s--Hajnal conjecture in this setting. Nguyen, Scott, and Seymour~\cite{NSS25} subsequently resolved the bounded-VC case completely: for every fixed \(d\), there is a constant \(c=c(d)>0\) such that every \(n\)-vertex graph with VC-dimension at most \(d\) contains a clique or an independent set of size at least \(n^{c}\), also see a recent improvement on the exponent~\cite{SWZ26}. Fox, Pach, and Suk~\cite{FPS21} also proved the Schur--Erd\H{o}s prediction \(r(k;m)=2^{\Theta_{k}(m)}\) for edge-colorings whose monochromatic neighborhood systems have bounded VC-dimension.

The sunflower problem under bounded VC-dimension was first studied by Fox, Pach, and Suk~\cite{FPS23}. They proved that an \(\ell\)-uniform family \(\cH\) with \(\VC(\cH)\le d\) contains an \(r\)-sunflower whenever \(\abs{\cH}\ge 2^{10\ell(dr)^{2\log^{*}\ell}}\). Balogh, Bernshteyn, Delcourt, Ferber, and Pham~\cite{BBDFP25} later obtained the substantially sharper sufficient condition \(\abs{\cH}>\bigl(Cr(\log d+\log^{*}\ell)\bigr)^{\ell}\) for an absolute constant \(C\), and also determined the sharp bound when \(d=1\). For general fixed \(d\), however, the exponential base still depends on \(\ell\). Our main result removes this final dependence.

\begin{theorem}\label{thm:main}
Let \(d\ge 1\), \(r\ge 2\), and \(\ell\ge 1\). If \(\cH\) is an \(\ell\)-bounded set system with \(\VC(\cH)\le d\) and no \(r\)-sunflower, then
\[
\abs{\cH}\le \sum_{k=0}^{\ell}(50dr)^{k}<2(50dr)^{\ell}.
\]
In particular, every \(\ell\)-uniform set system \(\cH\) with \(\VC(\cH)\le d\) and \(\abs{\cH}>(50dr)^{\ell}\) contains an \(r\)-sunflower.
\end{theorem}

\section{Proof of main result}

\subsection{Notation and Tools}
For a family \(\cG\) on a ground set \(X\) and \(T\subseteq X\), write
\[
d_{\cG}(T)=\abs{\{G\in\cG:T\subseteq G\}}.
\]
Thus \(d_{\cG}(T)\) is the number of members of \(\cG\) that contain \(T\).
For \(E\subseteq X\), the family \(\{G\cap E:G\in\cG\}\) is called the \emph{trace} of \(\cG\) on \(E\). We write \(\nu(\cG)\) for the matching number of \(\cG\), that is, the maximum number of pairwise disjoint members in \(\cG\).

Let \(0<q<1\). A nonempty family \(\cG\) is \emph{\(q\)-spread} if \(d_{\cG}(T)\le q^{\abs{T}}\abs{\cG}\) for every nonempty \(T\subseteq X\). In other words, the proportion of members containing any prescribed set \(T\) is at most \(q^{\abs{T}}\). If \(G\) is chosen uniformly from \(\cG\), meaning that each member is chosen with probability \(\frac{1}{\abs{\cG}}\), we write
\[
\mu_{x}=\mathbb{P}(x\in G)=\frac{d_{\cG}(\{x\})}{\abs{\cG}}.
\]
Thus \(\mu_{x}\) is simply the probability that a random member contains \(x\). For a \(q\)-spread family, the case \(T=\{x\}\) gives \(\mu_{x}\le q\).

We use the following standard form of the Sauer--Shelah lemma~\cite{Sauer72,Shelah72,VC71}.

\begin{lemma}\label{lem:sauer}
Let \(\cG\) be a set system on \(X\) with \(\VC(\cG)\le d\), and let \(E\subseteq X\) have size \(m\ge d\ge 1\). Then
\[
\abs{\{G\cap E:G\in\cG\}}\le \sum_{i=0}^{d}\binom{m}{i}\le \left(\frac{em}{d}\right)^{d}.
\]
\end{lemma}

We shall use some facts about entropy. If a discrete random variable \(Y\) takes only finitely many values, its \emph{support} is the set of its possible values, namely
\[
\operatorname{supp}(Y)=\{y:\mathbb{P}(Y=y)>0\}.
\]
Writing \(p_{y}=\mathbb{P}(Y=y)\), the \emph{entropy} of \(Y\) is defined as
\[
H(Y)=\sum_{y\in\operatorname{supp}(Y)}p_{y}\log\frac{1}{p_{y}}.
\]
Thus entropy is the average of \(\log\frac{1}{p_{y}}\), where each possible value \(y\) is given its probability \(p_{y}\). Unless a base is displayed explicitly, all logarithms in the proof are natural. The entropy cannot exceed the logarithm of the number of values that the random variable can take. Indeed, by the concavity of the logarithm,
\begin{equation}\label{eq:concavity}
H(Y)\le \log\left(\sum_{y\in\operatorname{supp}(Y)}p_{y}\frac{1}{p_{y}}\right)
=\log\abs{\operatorname{supp}(Y)}.
\end{equation}

We also use the following standard Caro--Wei bound for the independence number of a graph, see~\cite{AS16}. Here \(\alpha(J)\) denotes the largest size of an independent set in \(J\), and \(\deg_{J}(v)\) denotes the degree of a vertex \(v\).

\begin{lemma}\label{lem:independence}
If \(J\) is a graph with \(N\) vertices and \(e(J)\) edges, then
\[
\alpha(J)\ge \sum_{v\in V(J)}\frac{1}{\deg_{J}(v)+1}\ge \frac{N^{2}}{N+2e(J)}.
\]
\end{lemma}

\subsection{Proof of Theorem~\ref{thm:main}}

We first bound the sum of the probabilities with which the elements of a fixed set \(E\) occur in a random member of \(\cG\). It perfectly combines the property of \(q\)-spread family and Sauer--Shelah lemma, which is the key ingredient of the proof.

\begin{lemma}\label{lem:local-entropy}
Let \(\cG\) be a \(q\)-spread family with \(\VC(\cG)\le d\), where \(0<q<1\) and \(d\ge 1\). Let \(G\) be chosen uniformly at random from \(\cG\), and put \(\mu_{x}=\mathbb{P}(x\in G)\). If \(E\subseteq X\) has size \(m\ge d\), then
\[
\sum_{x\in E}\mu_{x}\le \frac{d\log\left(\frac{em}{d}\right)}{\log\left(\frac{1}{q}\right)}.
\]
\end{lemma}

\begin{proof}[Proof of Lemma~\ref{lem:local-entropy}]
Set \(R=G\cap E\). Thus \(R\) records exactly which elements of \(E\) occur in the random set \(G\). If \(A\subseteq E\) is a possible value of \(R\), then the event \(R=A\) implies \(A\subseteq G\), and therefore
\[
\mathbb{P}(R=A)\le \mathbb{P}(A\subseteq G)\le q^{\abs{A}}.
\]
For \(A\ne\varnothing\), the second inequality is exactly the spread condition. For \(A=\varnothing\), it is the trivial inequality \(\mathbb{P}(R=\varnothing)\le 1=q^{0}\). Taking logarithms gives
\[
\log\frac{1}{\mathbb{P}(R=A)}\ge \abs{A}\log\frac{1}{q}.
\]
We multiply this inequality by \(\mathbb{P}(R=A)\) and sum over all possible values \(A\) of \(R\). By the definition of entropy,
\[
H(R)\ge \log\left(\frac{1}{q}\right)\sum_{A\in\operatorname{supp}(R)}\mathbb{P}(R=A)\abs{A}
=\log\left(\frac{1}{q}\right)\sum_{x\in E}\mu_{x}.
\]
To see the last equality directly, the middle sum is the average size of \(R\). An element \(x\in E\) contributes one to this size precisely when \(x\in G\), which occurs with probability \(\mu_{x}\).

It remains to bound the same entropy from above. Since every member of \(\cG\) is chosen with positive probability, the possible values of \(R\) are exactly the sets \(G\cap E\) with \(G\in\cG\). Hence Lemma~\ref{lem:sauer} and the general entropy bound~\eqref{eq:concavity} give
\[
H(R)\le\log\abs{\operatorname{supp}(R)}
=\log\abs{\{G\cap E:G\in\cG\}}
\le d\log\left(\frac{em}{d}\right).
\]
Comparing the lower and upper bounds for \(H(R)\) proves the lemma.
\end{proof}

We next bound the sum of the squared inclusion probabilities based on lemma~\ref{lem:local-entropy}.

\begin{lemma}\label{lem:square-sum}
Let \(\cG\) be \(q\)-spread with \(\VC(\cG)\le d\), where \(d\ge 1\) and \(0<q\le\frac{1}{2}\). If \(G\) is chosen uniformly at random from \(\cG\) and \(\mu_{x}=\mathbb{P}(x\in G)\), then
\[
\sum_{x\in X}\mu_{x}^{2}<49dq.
\]
\end{lemma}

\begin{proof}[Proof of Lemma~\ref{lem:square-sum}]
For every \(x\in X\), the spread condition applied to \(\{x\}\) gives \(\mu_{x}\le q\). We separate the elements according to the size of this probability. For each integer \(j\ge 0\), let
\[
E_{j}=\{x\in X:2^{-j-1}q<\mu_{x}\le 2^{-j}q\}
\]
and set \(m_{j}=\abs{E_{j}}\). The sets \(E_{j}\) partition the elements for which \(\mu_{x}>0\), elements with \(\mu_{x}=0\) do not affect the desired sum. Moreover, every \(x\in E_{j}\) satisfies \(\mu_{x}^{2}\le 4^{-j}q^{2}\). It therefore remains to control the number \(m_{j}\) of elements in each group.

Fix \(j\ge 0\) with \(m_{j}\ge d\). Applying Lemma~\ref{lem:local-entropy} to \(E_{j}\), and using \(\mu_{x}>2^{-j-1}q\) for every \(x\in E_{j}\), gives
\[
m_{j}2^{-j-1}q\log\left(\frac{1}{q}\right)
<\sum_{x\in E_{j}}\mu_{x}\log\left(\frac{1}{q}\right)
\le d\log\left(\frac{em_{j}}{d}\right).
\]
Put \(L=\log\left(\frac{1}{q}\right)\) and \(u_{j}=\frac{m_{j}2^{-j-1}q}{d}\). The definition of \(u_{j}\) gives \(\frac{m_{j}}{d}=\frac{u_{j}2^{j+1}}{q}\), and hence
\[
\log\left(\frac{em_{j}}{d}\right)=1+\log u_{j}+(j+1)\log 2+L.
\]
Thus the preceding inequality becomes
\begin{equation}\label{eq:uj}
u_{j}L<L+(j+1)\log 2+1+\log u_{j}.
\end{equation}

\begin{claim}\label{claim:uj}
If \(m_{j}\ge d\), then \(u_{j}\le 4(j+2)\).
\end{claim}

\begin{poc}
Suppose that \(u_{j}>4(j+2)\). Since \(j\ge 0\), this implies \(u_{j}>8\), and in particular \(\log u_{j}>0\). Also, \(q\le\frac{1}{2}\) gives \(L\ge\log 2\). After dividing \eqref{eq:uj} by \(L\), we use \(\frac{(j+1)\log 2}{L}\le j+1\), \(\frac{1}{L}\le\frac{1}{\log 2}<2\), and \(\frac{\log u_{j}}{L}\le\log_{2}u_{j}\). It follows that
\[
u_{j}<j+4+\log_{2}u_{j}.
\]
The elementary inequality \(\log_{2}u\le\frac{u}{2}\) holds for \(u\ge 8\), so \(u_{j}<2j+8\). On the other hand, our assumption gives \(u_{j}>4j+8\ge 2j+8\), a contradiction.
\end{poc}

Returning to the definition of \(u_{j}\), Claim~\ref{claim:uj} gives
\[
m_{j}=\frac{du_{j}2^{j+1}}{q}\le \frac{8d(j+2)2^{j}}{q}
\]
whenever \(m_{j}\ge d\). Since every element of \(E_{j}\) contributes at most \(4^{-j}q^{2}\), such a group contributes at most
\[
\sum_{x\in E_{j}}\mu_{x}^{2}\le m_{j}4^{-j}q^{2}\le 8dq(j+2)2^{-j}.
\]
Summing over all groups with \(m_{j}\ge d\), and using \(\sum_{j\ge 0}(j+2)2^{-j}=6\), we obtain a total contribution of at most \(48dq\).

If instead \(m_{j}<d\), then the same pointwise bound gives
\[
\sum_{x\in E_{j}}\mu_{x}^{2}\le m_{j}4^{-j}q^{2}<d4^{-j}q^{2}.
\]
The total contribution of these groups is therefore less than \(dq^{2}\sum_{j\ge 0}4^{-j}=\frac{4}{3}dq^{2}\le\frac{2}{3}dq\), where the last inequality uses \(q\le\frac{1}{2}\). Adding the two contributions yields
\[
\sum_{x\in X}\mu_{x}^{2}<48dq+\frac{2}{3}dq=\frac{146}{3}dq<49dq,
\]
as required.
\end{proof}

Lemma~\ref{lem:square-sum} controls intersections in an averaged sense. Indeed, for a fixed \(x\in X\), there are exactly \(d_{\cG}(\{x\})^{2}\) ordered pairs \((A,B)\in\cG^{2}\) for which \(x\in A\cap B\). Summing this count over \(x\) will allow us to find a large pairwise disjoint subfamily.

\begin{lemma}\label{lem:matching}
Let \(\cG\) be a finite family of nonempty sets. If \(\cG\) is \(q\)-spread, \(\VC(\cG)\le d\), \(d\ge 1\), and \(0<q\le\frac{1}{2}\), then
\(
\nu(\cG)>\frac{1}{49dq}.
\)
\end{lemma}

\begin{proof}[Proof of Lemma~\ref{lem:matching}]
Define a graph \(J\) whose vertices are the members of \(\cG\), with two distinct members joined by an edge exactly when they intersect. Thus an independent set in \(J\) is the same as a pairwise disjoint subfamily of \(\cG\).

Write \(N=\abs{\cG}\), choose a member \(G\) uniformly from \(\cG\), and put \(\mu_{x}=\mathbb{P}(x\in G)=\frac{d_{\cG}(\{x\})}{N}\). Consider the sum of \(\abs{A\cap B}\) over all ordered pairs \((A,B)\in\cG^{2}\). Since every member of \(\cG\) is nonempty, each of the \(N\) pairs \((A,A)\) contributes at least one. Moreover, every edge \(\{A,B\}\) of \(J\) gives the two ordered pairs \((A,B)\) and \((B,A)\), each of which also contributes at least one. It follows that
\[
N+2e(J)\le\sum_{A,B\in\cG}\abs{A\cap B}.
\]
We now count the sum on the right by first fixing \(x\in X\). There are \(d_{\cG}(\{x\})\) choices of \(A\) containing \(x\), and the same number of choices for \(B\). Thus \(x\) is counted in exactly \(d_{\cG}(\{x\})^{2}\) ordered pairs, and consequently
\begin{equation*}
\sum_{A,B\in\cG}\abs{A\cap B}
=\sum_{x\in X}d_{\cG}(\{x\})^{2}
=N^{2}\sum_{x\in X}\mu_{x}^{2}
<49dqN^{2},
\end{equation*}
where the last inequality follows from Lemma~\ref{lem:square-sum}. Combining the two estimates and applying Lemma~\ref{lem:independence}, we obtain
\[
\alpha(J)\ge\frac{N^{2}}{N+2e(J)}>\frac{1}{49dq}.
\]
Since independent sets in \(J\) are precisely the pairwise disjoint subfamilies of \(\cG\), we have \(\nu(\cG)=\alpha(J)\), which proves the lemma.
\end{proof}

We next show how to obtain a spread family from a sufficiently large uniform family. We select a set \(C\) that is contained in many members of the original family, but for which no proper extension is contained in a comparably large proportion of the members. After \(C\) is removed, this maximality becomes exactly the spread condition.

\begin{lemma}\label{lem:maximal-kernel}
Let \(0<q<1\), and let \(\cH\subseteq\binom{X}{\ell}\) satisfy \(\abs{\cH}>q^{-\ell}\). Then there is a set \(C\subseteq X\) with \(\abs{C}<\ell\) such that
\[
\cG=\{H\setminus C:H\in\cH,\ C\subseteq H\}
\]
is a nonempty \(q\)-spread, \((\ell-\abs{C})\)-uniform family with \(\VC(\cG)\le\VC(\cH)\).
\end{lemma}

\begin{proof}[Proof of Lemma~\ref{lem:maximal-kernel}]
Choose an inclusion-maximal set \(C\subseteq X\) satisfying
\begin{equation}\label{eq:dense-kernel}
d_{\cH}(C)\ge \abs{\cH}q^{\abs{C}}.
\end{equation}
Such a choice is possible because \(X\) is finite and \(C=\varnothing\) satisfies \eqref{eq:dense-kernel}.

The right-hand side of \eqref{eq:dense-kernel} is positive, so \(C\) is contained in at least one member of \(\cH\). In particular, \(\abs{C}\le\ell\). Suppose that \(\abs{C}=\ell\). Since \(\cH\) is \(\ell\)-uniform, the only member of \(\cH\) that can contain \(C\) is \(C\) itself, and hence \(d_{\cH}(C)=1\). On the other hand,
\[
d_{\cH}(C)\ge\abs{\cH}q^{\ell}>q^{-\ell}q^{\ell}=1,
\]
a contradiction. Consequently, \(\abs{C}<\ell\).

Every member of \(\cH\) containing \(C\) produces, after \(C\) is removed, a set of size \(\ell-\abs{C}\). Moreover, the original member can be recovered by adding \(C\) back, so distinct members of \(\cH\) produce distinct members of \(\cG\). Thus \(\cG\) is nonempty and \((\ell-\abs{C})\)-uniform, and
\(
\abs{\cG}=d_{\cH}(C).
\)

We next verify the spread condition. Let \(\varnothing\ne T\subseteq X\setminus C\). If \(C\cup T\) also satisfied \eqref{eq:dense-kernel}, it would be a proper extension of \(C\) satisfying the same condition, contrary to the maximal choice of \(C\). Hence
\[
d_{\cH}(C\cup T)<\abs{\cH}q^{\abs{C}+\abs{T}}.
\]
A member of \(\cG\) contains \(T\) exactly when the corresponding member of \(\cH\) contains \(C\cup T\). Using \eqref{eq:dense-kernel}, we therefore obtain
\[
d_{\cG}(T)=d_{\cH}(C\cup T)
<\abs{\cH}q^{\abs{C}}q^{\abs{T}}
\le\abs{\cG}q^{\abs{T}}.
\]
This is precisely the \(q\)-spread condition.

It remains to compare the VC-dimensions. Suppose that \(D\subseteq X\setminus C\) is shattered by \(\cG\). For every \(A\subseteq D\), there is a member \(G\in\cG\) such that \(G\cap D=A\). By the definition of \(\cG\), there is an \(H\in\cH\) containing \(C\) for which \(G=H\setminus C\). Since \(D\cap C=\varnothing\),
\[
H\cap D=(H\setminus C)\cap D=G\cap D=A.
\]
Thus every subset of \(D\) also occurs as the intersection of \(D\) with a member of \(\cH\), so \(D\) is shattered by \(\cH\). We conclude that \(\VC(\cG)\le\VC(\cH)\).
\end{proof}

\begin{proof}[Proof of Theorem~\ref{thm:main}]
We first prove the assertion for an \(\ell\)-uniform family. Set \(Q=50dr\) and \(q=Q^{-1}\). Since \(d\ge 1\) and \(r\ge 2\), we have \(q\le\frac{1}{2}\).

Suppose that \(\cH\subseteq\binom{X}{\ell}\), that \(\VC(\cH)\le d\), and that \(\abs{\cH}>Q^{\ell}=q^{-\ell}\). By Lemma~\ref{lem:maximal-kernel}, there is a set \(C\) such that
\[
\cG=\{H\setminus C:H\in\cH,\ C\subseteq H\}
\]
is a nonempty \(q\)-spread family with VC-dimension at most \(d\). Lemma~\ref{lem:matching} gives
\[
\nu(\cG)>\frac{1}{49dq}=\frac{50r}{49}>r.
\]
In particular, \(\cG\) contains \(r\) pairwise disjoint members \(G_{1},\ldots,G_{r}\). For each \(i\), the definition of \(\cG\) gives a member \(H_{i}=C\cup G_{i}\in\cH\). Every \(G_{i}\) is disjoint from \(C\), and \(G_{i}\cap G_{j}=\varnothing\) whenever \(i\ne j\). Consequently,
\[
H_{i}\cap H_{j}=(C\cup G_{i})\cap(C\cup G_{j})=C
\]
for every \(i\ne j\). Thus \(H_{1},\ldots,H_{r}\) form an \(r\)-sunflower with kernel \(C\). We have proved that every \(\ell\)-uniform family with no \(r\)-sunflower has at most \(Q^{\ell}\) members.

Now let \(\cH\) be \(\ell\)-bounded, and partition it according to the sizes of its members. For \(0\le k\le\ell\), let \(\cH_{k}=\cH\cap\binom{X}{k}\). Each \(\cH_{k}\) is a subfamily of \(\cH\). Removing members cannot create a shattered set, so \(\VC(\cH_{k})\le d\). Likewise, an \(r\)-sunflower in \(\cH_{k}\) would also be an \(r\)-sunflower in \(\cH\), and hence no \(\cH_{k}\) contains one. The uniform case therefore gives \(\abs{\cH_{k}}\le Q^{k}\) for \(1\le k\le\ell\). The family \(\cH_{0}\) contains at most the empty set, so \(\abs{\cH_{0}}\le 1=Q^{0}\). It follows that
\[
\abs{\cH}
=\sum_{k=0}^{\ell}\abs{\cH_{k}}
\le\sum_{k=0}^{\ell}Q^{k}
=\frac{Q^{\ell+1}-1}{Q-1}
<2Q^{\ell},
\]
where the final inequality uses \(Q\ge 2\). Substituting \(Q=50dr\) completes the proof.
\end{proof}

\bibliographystyle{abbrv}
\bibliography{bib}

\end{document}